\documentclass[12pt,a4paper,reqno]{amsart}

\usepackage[active]{srcltx}
\usepackage{a4wide}
\usepackage{amssymb}
\usepackage{graphicx}
\usepackage{float}
\usepackage{colortbl}

\theoremstyle{plain}% default
\newtheorem{theorem}{Theorem}[section]
\newtheorem{maintheorem}{Theorem}
\newtheorem{lemma}[theorem]{Lemma}
\newtheorem{proposition}[theorem]{Proposition}
\newtheorem{corollary}[theorem]{Corollary}

\theoremstyle{definition}
\newtheorem*{condition}{Condition}

\theoremstyle{remark}
\newtheorem{remark}[theorem]{Remark}%[section]

\def\R{\ensuremath{\mathbb R}}
\def\N{\ensuremath{\mathbb N}}

\def\I{\ensuremath{{\bf 1}}}
\def\e{\ensuremath{\text e}}

\def\B{\ensuremath{\mathcal BC}}

\def\P{\ensuremath{\mathcal P}}
\def\F{\ensuremath{\mathcal F}}

\def\QQ{\ensuremath{\mathcal Q}}

\numberwithin{equation}{section}

\begin{document}

  %Included for Gather Purpose only:
  %GATHER{ExtremeDynamics2.bib}
  %input "ExtremeDynamics2.bib"

\title[Extreme values
for Benedicks-Carleson quadratic maps]{Extreme values for
Benedicks-Carleson\\ quadratic maps}

\author[A. C. M. Freitas]{Ana Cristina Moreira Freitas}
\address{Ana Cristina Moreira Freitas\\ Centro de Matem\'{a}tica \&
Faculdade de Economia da Universidade do Porto\\ Rua Dr. Roberto Frias \\
4200-464 Porto\\ Portugal} \email{amoreira@fep.up.pt}

\author[J. M. Freitas]{Jorge Milhazes Freitas}
\address{Jorge Milhazes Freitas\\ Centro de Matem\'{a}tica da Universidade do Porto\\ Rua do
Campo Alegre 687\\ 4169-007 Porto\\ Portugal}
\email{jmfreita@fc.up.pt}
\urladdr{http://www.fc.up.pt/pessoas/jmfreita}

\date{\today}
\thanks{Work partially supported by FCT through CMUP and POCI/MAT/61237/2004.}

\begin{abstract}
We consider the quadratic family of maps given by $f_{a}(x)=1-a x^2$
with $x\in [-1,1]$, where $\emph a$ is a Benedicks-Carleson
parameter. For each of these chaotic dynamical systems we study the
extreme value distribution of the stationary stochastic processes
$X_0,X_1,\ldots$, given by $X_{n}=f_a^n$, for every integer
$n\geq0$, where each random variable $X_n$ is distributed according
to the unique absolutely continuous, invariant probability of $f_a$.
Using techniques developed by Benedicks and Carleson, we show that
the limiting distribution of $M_n=\max\{X_0,\ldots,X_{n-1}\}$ is the
same as that which would apply if the sequence $X_0,X_1,\ldots$ was
independent and identically distributed. This result allows us to
conclude that the asymptotic distribution of $M_n$ is of Type III
(Weibull).
\end{abstract}

\maketitle

\setcounter{tocdepth}{1}

%\tableofcontents %\clearpage

\section{Introduction}\label{sec:introduction}

The dynamical systems we consider in this work are the quadratic
maps given by $f_a(x)=1-ax^2$ on $I=[-1,1]$, with $a\in\B$, where
$\B$ is the Benedicks-Carleson parameter set introduced in
\cite{BC85}. The set $\B$ has positive Lebesgue measure and is built
in such a way that, for every $a\in\B$, the Collet-Eckmann condition
holds: there is exponential growth of
the derivative of $f_a$ along the critical orbit, i.e., there is
$c>0$ such that
\[
\left|\left(f_a^n\right)'(f_a(0))\right|\geq \e^{cn},
\]
for all $n\in\N$. This property guarantees not only the
non-existence of an attracting periodic orbit but also the existence
of an ergodic $f_a$-invariant probability measure $\mu_a$ that is
absolutely continuous with respect to Lebesgue measure on $[-1,1]$.
In fact, these Benedicks-Carleson systems are chaotic and highly
sensitive on initial conditions. Actually, after some iterates, the
behavior of most orbits becomes erratic and distributed on the set
$[-1,1]$ according to the invariant measure $\mu_a.$
 Hence, it is meaningful to study the statistical
properties of the orbits of these systems and, here, we are
particularly concerned with their extreme type behavior.

A natural way to build a stationary stochastic process associated to
$f_a$, for some $a\in\B$, is to consider the random variable (r.v.)
$X_0$ defined on the probability space $([-1,1],\mu_a)$, taking
values on $[-1,1]$ with distribution function (d.f.)
$G_a(x)=\mu_a\{(-\infty,x]\cap[-1,1]\}$ and then iterate it by
$f_a$, i.e., define
\begin{equation}
\label{eq:def-process} X_n=f_a(X_{n-1})=f_a^{n}(X_0),\; \mbox{ for
each $n\in\N$}.
\end{equation}
This way, we obtain a stationary stochastic process
$X_0,X_1,X_2,\ldots$ with common marginal d.f. given by
$G_a(x)=\mu_a\left\{ X_0\leq x\right\}$. The stationarity results
from the $f_a$-invariance of the probability measure $\mu_a$ (see
for example \cite[Section~15.4]{KT66}). Our goal is to study the
asymptotic distribution of the partial maximum
\begin{equation}
\label{eq:def-maximum} M_n=\max\left\{X_0, X_1,\ldots,
X_{n-1}\right\},
\end{equation}
when properly normalized.

The study of distributional properties of the higher order
statistics of a sample, like the maximum of the sample is
the purpose of Extreme Value Theory. Its major classical theorem
asserts that there are only three types of non-degenerate asymptotic
distributions for the maximum of an independent and identically
distributed (i.i.d.) sample under linear normalization.

The main result of this work states that the limiting law of $M_n$
is the same as if $X_0, X_1,\ldots$ were independent with the same
d.f. $G_a$. In fact, we verify that under appropriate normalization
the asymptotic distribution of $M_n$ is of type III (Weibull). As
usual, we denote by $G_a^{-1}$ the generalized inverse of the d.f.
$G_a$, which is to say that $G_a^{-1}(y):=\inf\{x:\,G_a(x)\geq y\}$.
\begin{maintheorem}
\label{th:A-tipo-1} For each $a\in\B$ and every stationary
stochastic process $\left(X_i\right)_{i\in\N_0}$ given by
\eqref{eq:def-process}, we have:
\[
P\{a_n(M_n-1)\leq x\}\rightarrow H(x)=\begin{cases}
\e^{-(-x)^{1/2}}&,x\leq0\\
1&,x>0
\end{cases},
\]
where $a_n=\left(1-G_a^{-1}\left(1-\frac1n\right)\right)^{-1}$.
\end{maintheorem}

Haiman \cite{Ha03} has obtained a similar asymptotic
result for the natural stochastic process associated with the tent
map. The arguments used here for
quadratic maps would allow us to obtain a different proof of
Haiman's Theorem.

Also, a study concerning extremes for dynamical systems, essentially
focusing the finite sample behavior of maxima, has already been done
by Balakrishnan, Nicolis and Nicolis in \cite{BNN95}.

More recently, Collet \cite{Co01} has studied the asymptotic
distribution of the partial maximum of the stochastic process $Y_1,
Y_2,\ldots$ given by $Y_n=-\log|\xi-f^n|$, where $\xi$ is a
reference point in the phase space and $f$ a $C^2$ non-uniformly
hyperbolic interval map. He obtained that the Gumbel's limiting law
applies for Lebesgue-a.e. point $\xi$ chosen in the phase space.
However, the critical point and its orbit are not included in this
full Lebesgue measure set of points.

\section{Motivation and Strategy}\label{sec:strategy}

The study of the limit behavior for maxima of a stationary process
can be reduced, under adequate conditions on the dependence
structure, to the Classical Extreme Value Theory for sequences of
i.i.d. random variables. Hence, to the stationary process $X_0,
X_1,\ldots$ we associate an independent sequence of r.v. denoted by
$Z_0, Z_1,\ldots$ with common d.f. given by $G_a(x)=P\left\{X_0\leq
x\right\}=\mu_a\left\{X_0\leq x \right\}$. We also set for each
$n\in\N$
\begin{equation}
\label{eq:def-maximum-iid}
\hat{M}_n=\max\left\{Z_0,\ldots,Z_{n-1}\right\}.
\end{equation}
Let us focus on the conditions that allow us to relate the
asymptotic distribution of $M_n$ with that of $\hat M_n$. Following
\cite{LLR83} we refer to these conditions as $D(u_n)$ and $D'(u_n)$,
where $u_n$ is a suitable sequence of thresholds converging to
$\max_{z\in[-1,1]}X_0(z)=1$, as $n$ goes to $\infty$, that will be
defined below. $D(u_n)$ imposes a certain type of distributional
mixing property. Essentially, it says that the dependence between
some special type of events fades away as they become more and more
apart in the time line. $D'(u_n)$ restricts the appearance of
clusters, that is, it makes the occurrence of consecutive
`exceedances' of the level $u_n$ an unlikely event.

Consider a sequence of stationary r.v. $Y_1,Y_2\ldots$ with common
d.f. $F$. We say that an \emph{exceedance} of the \emph{level} $u_n$
occurs at time $i$ if $Y_i>u_n$. The probability of such an
exceedance is $1-F(u_n)$ and so the mean value of the number of
exceedances occurring up to $n$ is $n(1-F(u_n))$. The sequences of
levels $u_n$ we consider are such that $n(1-F(u_n))\rightarrow\tau$
as $n\rightarrow\infty$, for some $\tau\geq0$, which means that, in
a time period of length $n$, the expected number of exceedances is
approximately $\tau$.

Observe that, in the case of our stationary process
$X_0,X_1,\ldots$, we have $X_j>u_n$ if and only if
$|X_{j-1}|<\sqrt{\frac{1-u_n}{a}}$. Since $u_n\to1$ as $n\to\infty$,
then, for large $n$, an exceedance of the level $u_n$ is always
preceded by a return to a tight vicinity of the critical point.
Hence, as one would expect, there is an intimate relation between
exceedances and deep returns to the vicinity of the critical point.

$D(u_n)$ is a type of mixing requirement specially adapted to
Extreme Value Theory. In this context, the events of interest are
those of the form $\{X_i\leq u\}$ and their intersections. Observe
that $\{M_n\leq u\}$ is just $\{X_0\leq u,\ldots,X_{n-1}\leq u\}$. A
natural mixing condition in this context is the following. Let
$G_{i_1,\ldots,i_n}(x_1,\ldots,x_n)$ denote the joint d.f. of
$X_{i_1},\ldots,X_{i_n}$ and set
$G_{i_1,\ldots,i_n}(u)=G_{i_1,\ldots,i_n}(u,\ldots,u)$.
\begin{condition}[$D(u_n)$]\label{cond:D}We say that $D(u_n)$ holds
for the sequence $X_0,X_1,X_2,\ldots$ if for any integers
$i_1<\ldots<i_p$ and $j_1<\ldots<j_k$ for which $j_1-i_p>m$, and any
large $n\in\N$,
\[
\left|G_{i_1,\ldots,i_p,j_1,\ldots,j_k}(u_n)-G_{i_1,\ldots,i_p}(u_n)
G_{j_1,\ldots,j_k}(u_n)\right|\leq \gamma(m),
\]
where $\gamma(m)\rightarrow0$ as $m\rightarrow\infty$.
\end{condition}

We remark that the actual definition of $D(u_n)$ appearing in
\cite[Section~3.2]{LLR83} is a weaker requirement. Following \cite{Co01}, we proposed, very recently, in \cite{FF2}, a reformulation of condition $D(u_n)$ that follows almost immediately
from the decay of correlations of the system and together with
$D'(un)$ still allows to relate the asymptotic distributions of
$M_n$ and $\hat M_n$.

\begin{condition}[$D'(u_n)$]\label{cond:D'un} We say that $D'(u_n)$
holds for the sequence $X_0,X_1,X_2,\ldots$ if
\begin{equation}
\label{eq:D'un} \lim_{k\rightarrow\infty}
\limsup_{n\rightarrow\infty}\,n\sum_{j=1}^{[n/k]}P\{X_0>u_n\text{
and }X_j>u_n\}=0.
\end{equation}
\end{condition}

The sequence $u_n$ is such that the average number of exceedances in
the time interval $\{0,\ldots,[n/k]\}$ is approximately $\tau/k$,
which goes to zero as $k\rightarrow\infty$. However, the exceedances
may have a tendency to be concentrated in the time period following
the first exceedance at time $0$. Condition \ref{eq:D'un} prevents
this from happening, i.e., forbids the concentration of exceedances
by bounding the probability of more than one exceedance in the time
interval $\{0,\ldots,[n/k]\}$. This guarantees that the exceedances
should appear scattered through the time period $\{0,\ldots,n-1\}$.

The special relevance of both these conditions is connected to
Theorem~3.5.2 of \cite{LLR83}: let $a_n$ and $b_n$ be sequences such
that $P\{a_n(\hat M_n-b_n)\leq x\}\rightarrow H(x)$ for some
non-degenerate d.f. $H$; if $D(u_n)$, $D'(u_n)$ hold for the
stationary sequence $X_1,X_2,\ldots$, when $u_n=x/a_n+b_n$ for each
$x$, then $P\{a_n( M_n-b_n)\leq x\}\rightarrow H(x)$. This means
that if we are able to show that conditions $D(u_n)$ and $D'(u_n)$
are valid for the stationary process $X_0,X_1,\ldots$, then $M_n$
and $\hat M_n$ share the same asymptotic distribution with the same
normalizing sequences. Consequently, our strategy to prove
Theorem~\ref{th:A-tipo-1} is the following:
\begin{itemize}

\item Compute the limiting distribution of $\hat M_n$ and the
associated normalizing sequences $a_n$ and $b_n$.

\item Show that conditions $D(u_n)$ and $D'(u_n)$ are
valid for the stochastic process \linebreak $X_0,X_1,X_2,\ldots$
defined in \eqref{eq:def-process}.

\end{itemize}
The rest of the paper is dedicated to the proof of these assertions
and is structured as follows. In Section~\ref{sec:BC-techniques}, we
describe the properties of the dynamical systems $f_a$, with $a\in
\B$, and the Benedicks-Carleson techniques. The validity of
condition $D(u_n)$ is a consequence of the very good mixing
properties of the systems considered here. Actually, it follows from
the fact that these systems possess a weak-Bernoulli generator (see
Section~\ref{subsec:weak-Bernoulli} and
Remark~\ref{rem:weak-bernoulli-Dun}). Then, in
Section~\ref{sec:domain-attraction-iid}, we study the asymptotic
behavior of the maximum in the i.i.d. case, identify the limiting
distribution of $\hat M_n$ and the respective normalizing sequences
$a_n$ and $b_n$. Hence, we are left with the burden of proving
$D'(u_n)$. In Section~\ref{sec:Depth Probability}, we use the
geometric properties of the systems to show
Proposition~\ref{prop:depth-probability-1} that paves the way for
the proof of $D'(u_n)$ which is finally established in
Section~\ref{sec:D'un}. In Section~\ref{sec:simulation}, we present
a small simulation study in order to compare the finite sample
behavior of the normalized $M_n$ with the asymptotic one.

\section{Properties of the Benedicks-Carleson
parameters}\label{sec:BC-techniques}

The Benedicks-Carleson Theorem (see \cite{BC85} or Section 2 of
\cite{BC91}) states that there exists a positive Lebesgue measure
set of parameters, $\B$, verifying
\begin{align}
&\mbox{there is $c>0$ ($c\approx \log2$)
such that $|Df_a^n(f_a(0))|\geq \e^{cn}$ for all $n\geq0$};
\tag{EG}\label{eq:EG}\\
&\text{there is a small $\alpha>0$ such that $|f_a^n(0)|\geq\e^
{-\alpha \sqrt n}$ for all $n\geq 1$}\tag{BA}\label{eq:BA}.
\end{align}

Before we describe the Benedicks-Carleson strategy let us have an
overview of its key ingredients. The \emph{critical region} is the
interval $(-\delta,\delta)$, where $\delta=\e^{-\Delta}>0$ is chosen
small but much larger than $2-a$. This region is partitioned into
the intervals
\[
(-\delta,\delta)=\bigcup_{m\geq\Delta} I_m,
\]
where $I_m=(\e^{-(m+1)},\e^{-m}]$ for $m>0$ and $I_m=-I_{-m}$ for
$m<0$; then each $I_m$ is further subdivided into $m^2$ intervals
$\{I_{m,j}\}$ of equal length inducing the partition $\P_0$ of
$[-1,1]$ into
\begin{equation}
\label{eq:partition} [-1,-\delta)\cup
\bigcup_{m,j}I_{m,j}\cup(\delta,1].
\end{equation}
Given $J\in\P$, let $nJ$ denote the interval $n$ times the length of
$J$ centered at $J$ and define $U_m:=(-\e^{-m},\e^{-m})$, for every
$m\in\N$. In order to study the growth of $Df_a^n(x)$, for
$x\in[-1,1]$ and $a\in\B$, the orbit is splitted in free periods and
bound periods. During the former we are certain that the orbit never
visits the critical region. The latter begin when the orbit returns
to the critical region and initiates a bound to the critical point,
accompanying its early iterates. The behavior of the derivative
during these periods is detailed in Subsections
\ref{subsec:free-period-estimates} and
\ref{subsec:bound-period-estimates}.

\subsection{Expansion outside the critical region}
\label{subsec:free-period-estimates} There is $c_0>0$ and $M_0\in\N$
such that
\begin{enumerate}

\item \label{item:free-period-M0-1d} If
$x,\ldots,f_a^{k-1}(x)\notin(-\delta,\delta)$ and $k\geq M_0$, then
$|Df_a^k(x)|\geq\e^{c_0k}$;

\item \label{item:free-period-return-1d} If
$x,\ldots,f_a^{k-1}(x)\notin(-\delta,\delta)$ and
$f_a^k(x)\in(-\delta,\delta)$, then $|Df_a^k(x)|\geq\e^{c_0k}$;

\item \label{item:free-period-1d} If
$x,\ldots,f_a^{k-1}(x)\notin(-\delta,\delta)$, then
$|Df_a^k(x)|\geq\delta\e^{c_0k}$.

\end{enumerate}

If we were capable of keeping the orbit of $x$ away from the
critical region then it would be in \emph{free period} forever and
the estimates above would apply. However, it is inevitable that
almost every $x\in[-1,1]$ makes a \emph{return} to the critical
region. We say that $n\in\N$ is a \emph{return time} of the orbit of
$x$ if $f_a^n(x)\in(-\delta,\delta)$. Every free period of $x$ ends
with a \emph{free return} to the critical region. We say that the
return had a \emph{depth} of $m\in\N$ if  $f_a^n(x)\in I_{\pm m}$,
which is equivalent to saying that $m=\left[-\log
|f_a^n(x)|\right]$. Once in the critical region, the orbit of $x$
initiates a binding with the critical point.

\subsection{Bound period definition and properties}
\label{subsec:bound-period-estimates} Let $\beta=14\alpha$. For
$x\in(-\delta,\delta)$ define $p(x)$ to be the largest integer $p$
such that
\begin{equation}
\label{eq:def-bound-period} |f_a^k(x)-f_a^k(0)|<\e^{-\beta k},\qquad
\forall k<p.
\end{equation}Then

\begin{enumerate}

\item \label{item:bound-period-bounds-1d}$\frac{1}{2}|m|\leq p(x)\leq3|m|$,
for each $ x\in I_m$;

\item \label{item:bound-period-derivative-1d} $|Df_a^p(x)|\geq
\e^{c'p}$, where $c'=\frac{1-4\beta}{3}>0$.

\end{enumerate}
The orbit of $x$ is said to be bound to the critical point during
the period $0\leq k<p$. We may assume that $p$ is constant on each
$I_{m,j}$. Note that during the bound period the orbit of $x$ may
return to the critical region. These instants are called \emph{bound
return times}.

Roughly speaking, the idea behind the proof of Benedicks-Carleson
Theorem is that while the orbit of the critical point is outside the
critical region we have expansion (see Subsection
\ref{subsec:free-period-estimates}); when it returns we have a
serious setback in the expansion but then, by continuity, the orbit
repeats its early history regaining expansion on account of
\eqref{eq:EG}. To arrange for the exponential growth of the
derivative along the critical orbit \eqref{eq:EG} one has to
guarantee that the losses at the returns are not too drastic; hence,
by parameter elimination, the basic assumption condition
\eqref{eq:BA} is imposed. The argument is mounted in a very
intricate induction scheme that guarantees both the conditions for
the parameters that survive the exclusions. The condition
\eqref{eq:EG} is usually known as the Collet-Eckmann condition and
it was introduced in \cite{CE83}.

\subsection{Bookkeeping, essential and inessential returns}
A sequence of partitions $\P_0\prec\P_1\prec\ldots$ is built so that
points in the same element of the partition $\P_n$ have the same
history up to time $n$. For a detailed description of the
construction of this sequence of partitions in the phase space
setting we refer to \cite[Section~4]{Fr05}. Here, we highlight some
of the main aspects of its construction.

For Lebesgue almost every $x\in I$, $\{x\}=\cap_{n\geq
0}\omega_n(x)$, where $\omega_n(x)$ is the element of
$\mathcal{P}_n$ containing $x$. For such $x$ there is a sequence
$t_1,t_2,\ldots$ corresponding to the instants when the orbit of $x$
experiences a free \emph{essential return situation}, which means
$I_{m,k}\subset f_a^{t_i}(\omega_{t_i-1}(x))$ for some $|m|\geq
\Delta$ and $1\leq k\leq m^2$. We have that
$\omega_n(x)=\omega_{t_{i-1}}(x)$, for every $t_{i-1}\leq n<t_i$ and
$f_a^{t_i}(\omega_{t_i}(x))=\omega_0(f^{t_i}(x))$, except for the
points at the two ends of $f_a^{t_i}(\omega_{t_{i-1}}(x))$ for which
it may occur an adjoining to the neighboring interval. If $t_i$ is
an essential return situation for $x$, then it is either an
\emph{essential return time} for $x$, which means that there exists
$m\geq\Delta$ and $1\leq k\leq m^2$ such that $I_{m,k}\subset
f_a^{t_i}(\omega_{t_i}(x))\subset 3 I_{m,k}$; or an \emph{escaping
time} for $x$, which is to say that $I_{(\Delta-1),1}\subset
f_a^{t_i}(\omega_{t_i}(x))\subset (\delta,1]$ or
$I_{-(\Delta-1),1}\subset f_a^{t_i}(\omega_{t_i}(x))\subset
[-1,-\delta)$, where $I_{\pm(\Delta-1),1}$ is the subinterval of
$I_{\pm(\Delta-1)}$ closest to $0$.

We remark that every point in $\omega\in\P_n$ has the same history
up to $n$, in the sense that they have the same free periods, return
to the critical region simultaneously, with the same depth and their
bound periods expire at the same time.

We say that $v$ is a free return time for $x$ of \emph{inessential}
type if $f_a^{v}(\omega_{v}(x))\subset 3I_{m,k}$, for some $|m|\geq
\Delta$ and $1\leq k\leq m^2$, but $f_a^{v}(\omega_{v}(x))$ is not
large enough to contain an interval $I_{m,k}$ for some $|m|\geq
\Delta$ and $1\leq k\leq m^2$.

\subsection{Distortion of the derivative}
\label{subsec:bounded-distortion} The sequence of partitions
described above is designed so that we have bounded distortion in
each element of the partition $\P_{n-1}$ up to time $n$. To be more
precise, consider $\omega\in\P_{n-1}$. There exists a constant $C$
independent of $\omega$, $n$ and the parameter $a$ such that for
every $x,y\in\omega$,
\begin{equation}
\label{eq:bounded-distortion}
 \frac{|Df_a^n(x)|}{|Df_a^n(y)|}\leq C.
\end{equation}
See \cite[Lemma~4.2]{Fr05} for a proof.

\subsection{Growth of returning and escaping components}
\label{subsec:growth-components} Let $t$ be an essential return time
for $\omega\in\P_{t}$, with $I_{m,k}\subset f_a^{t}(\omega)\subset 3
I_{m,k}$ for some $m\geq\Delta$ and $1\leq k\leq m^2$. If $n$ is the
next free return situation for $\omega$ (either essential or
inessential) then
\begin{equation}
\label{eq:return-after-essential} \left|f_a^{n}(\omega)\right|\geq
\e^{c_0q}\e^{-5\beta|m|},
\end{equation}
where $q=n-(t+p)$. See \cite[Lemma~4.1]{Fr05}.

Suppose that $\omega\in\mathcal{P}_t$ is an escape component. Then,
in the next return situation for $\omega$, at time $t_1$, we have
that
\begin{equation}
\label{eq:return-after-escape} \left|f_a^{t_1}(\omega)\right|\geq
\e^{-\beta\Delta}.
\end{equation}
See \cite{MS93}, \cite[Lemma 4.2]{Fr06} or \cite[Lemma~5.1]{Mo92}
for a proof of a similar statement in the space of parameters.

\subsection{Existence of absolutely continuous invariant measures}
\label{subsection:existence-inv-measures} For every $a\in\B$, the
quadratic map $f_a$ has an invariant probability measure $\mu_a$
that is absolutely continuous with respect to Lebesgue measure on
$[-1,1]$. The existence of absolutely continuous invariant measures
(a.c.i.m) for a positive Lebesgue measure set of parameters was
first proved by Jakobson in \cite{Ja81} and others followed. See,
for example, \cite{CE83}, \cite{BC85}, \cite{No85}, \cite{Ry88},
\cite{BY92} and \cite{Yo92}.

The a.c.i.m. $\mu_a=\rho_a dx$ has the following properties:
\begin{enumerate}

\item \label{item:acim-uniqueness} $\mu_a$ is the only a.c.i.m. of
$f_a$;

\item \label{item:acim-exact} $(f_a,\mu_a)$ is exact;

\item \label{item:acim-density-1}$\rho_a=\rho_1^a+\rho_2^a$, where
$\rho_1^a$ has bounded variation and
$0\leq\rho_2^a(x)\leq\mbox{const}
\sum_{j=1}^\infty\frac{(1.9)^{-j}}{\sqrt{|x-f_a^j(0)|}}$;

\item \label{item:acim-density-2} The support of $\mu_a$ is
$[f_a^2(0),f_a(0)]$ and $\inf_{x\in[f_a^2(0),f_a(0)]}\rho_a(x)>0$.

\end{enumerate}
The proof of these statements can be found in \cite[Theorems 1 and
2]{Yo92}.

\subsection{Decay of correlations and Central Limit Theorem}
\label{subsec:decay-corr-CLT} The Benedicks-Carleson quadratic maps
have good statistical behavior. In fact, L. S. Young proved that
these maps have exponential decay of correlations and satisfy the
Central Limit Theorem (\cite[Theorems 3 and 4]{Yo92}). This was also
obtained by Keller and Nowicki in \cite{KN92}. To be more precise,
for every $a\in \B$, there exists $\varsigma\in(0,1)$ such that for
all $\varphi,\psi:[-1,1]\rightarrow\R$ with bounded variation, there
is $C=C(\varphi,\psi)$ such that
\begin{equation}
\label{eq:decay-correlations} \left| \int\varphi\cdot(\psi\circ
f_a^n)d\mu_a-\int\varphi d\mu_a\int\psi d\mu_a\right|\leq
C\varsigma^n,\quad\forall n\geq 0.
\end{equation}
Moreover, if $\int\varphi d\mu_a=0$ then for every $x\in\R$ we have
\begin{equation}
\label{eq:CLT}
\mu_a\left\{\frac1{\sqrt{n}}\sum_{i=0}^{n-1}\varphi\circ f_a^i\leq
x\right\}\xrightarrow[n\rightarrow\infty]{}\Phi(x/\sigma),
\end{equation}
where we are assuming that
$\sigma:=\lim_{n\rightarrow\infty}\frac1{\sqrt{n}}\left[\int\left(
\sum_{i=0}^{n-1}\varphi\circ f_a^i\right)^2d\mu_a\right]^{1/2}>0$
and $\Phi(\cdot)$ denotes the $N(0,1)$ d.f.

\subsection{Exponential weak-Bernoulli
mixing}\label{subsec:weak-Bernoulli} Keller \cite{Ke94} has obtained
a result even sharper than \eqref{eq:decay-correlations}. Consider
the partition of $[-1,1]$ given by
$\QQ=\left\{[-1,0),[0,1]\right\}$. Also, for integers $k<l$, denote
by $\QQ_k^l$ the join of partitions $\bigvee_{i=k}^lf_a^{-i}\QQ$ and
by $\F_k^l$ the $\sigma$-algebra generated by $\QQ_k^l$. According
to \cite{Ke94} the partition $\QQ$
 is a weak-Bernoulli generator for
every $f_a$ with $a\in\B$. This means that the $\sigma$-algebra
$\F_0^\infty$ coincides, up to sets of Lebesgue measure $0$, with
the Borel $\sigma$-algebra of sets in $[-1,1]$ and that
\[
\beta_m(f_a,\QQ,\mu_a)\rightarrow0,\; \mbox{ as }m\rightarrow\infty,
\]
where
\begin{equation*}
\begin{split}
\beta_m(f_a,\QQ,\mu_a):&=2\sup_{k>0}\int\sup\left\{
\left|\mu_a(A|\F_0^k)-\mu_a(A)\right|:\,A\in\F_{k+m}^\infty\right\}d\mu_a\\
&=\sup_{k\geq1,\,L\geq1}\sum_{A\in\QQ_0^k,\,B\in\QQ_{k+m}^{k+m+L}}
\left|\mu_a(A\cap B)-\mu_a(A)\mu_a(B)\right|.
\end{split}
\end{equation*}
In fact, \cite[Theorem~1]{Ke94} states that there are constants
$C>0$ and $0<r<1$ such that
\begin{equation}
\label{eq:weak-bernoulli-exponential} \beta_m(f_a,\QQ,\mu_a)\leq
Cr^m
\end{equation}
 for all $m\in\N$.

\begin{remark}
\label{rem:weak-bernoulli-Dun} We observe that if we refine the
partition $\QQ$ by adding one point so that $\{X_0>u\}\in\F_0$,
where $\F_0$ is the $\sigma$-algebra generated by $\QQ$, then
Keller's argument still holds with the same type of estimate as in
\eqref{eq:weak-bernoulli-exponential}. As a consequence, condition
$D(u_n)$ is true for every considered sequence $u_n$.

\end{remark}

\section{Domain of attraction of the associated i.i.d. process}
\label{sec:domain-attraction-iid} We recall that to every stationary
stochastic process $X_0,X_1,X_2,\ldots$ defined in
\eqref{eq:def-process} we associated an i.i.d. sequence of r.v.
$Z_0,Z_1,Z_2,\ldots$ with common d.f. given by $G_a(x)=P\{X_0\leq
x\}=\mu_a\{(-\infty,x]\cap[-1,1]\}$ (see
Section~\ref{sec:strategy}). In this Section we will determine the
domain of attraction corresponding to the d.f. $G_a$, i.e., we will
compute the limiting distribution of $\hat M_n$, defined in
\eqref{eq:def-maximum-iid}, when properly normalized. For that
purpose one must look at the tail behavior of $1-G_a(x)$ as $x$ gets
closer to $\sup_{y\in\R}\left\{G_a(y)<1\right\}=1$. According to
Section~\ref{subsection:existence-inv-measures}~
\eqref{item:acim-density-1}, if $z$ is close to $1$ we may write
$\mu_a\left\{(z,1]\right\}\asymp\sqrt{1-z})$, in the sense that
$\frac{\mu_a\left\{(z,1]\right\}}{\sqrt{1-z}}\rightarrow c$ for some
$c>0$, as $z\rightarrow1$. Hence, for $s>0$ sufficiently close to
$0$ we have:
\begin{equation}
\label{eq:tail-g1} 1-G_a\left(1-s\right)\asymp\sqrt{1-(1-s)}
\asymp\sqrt s,
\end{equation}
which means that
$\lim_{s\rightarrow0^+}\frac{1-G_a\left(1-s\right)}{\sqrt s}>0$. At
this point, we apply \cite[Theorem~1.6.2]{LLR83} to obtain that
$G_a$, in this case, belongs to the domain of attraction of type III
(Weibull) with parameter $1/2$, since for every $x>0$
\[
\lim_{h\rightarrow0^+}\frac{1-G_a\left(1-xh\right)}
{1-G_a\left(1-h\right)}=\lim_{h\rightarrow0^+}\frac{\sqrt{xh}}
{\sqrt{h}}=x^{1/2}.
\]
Moreover, according to \cite[Corollary~1.6.3]{LLR83}, if we consider
the sequences defined for each $n\in\N$ by $b_n=1$ and $a_n=
\left(1-G_a^{-1}(1-1/n)\right)^{-1}$, where
$G_a^{-1}(y)=\inf\{x:\,G_a(x)\geq y\}$, then
\[
P\left\{a_n(\hat M_n-b_n)\leq x\right\}\rightarrow
H(x)=\begin{cases}
\e^{-(-x)^{1/2}}&,x\leq0\\
1&,x>0
\end{cases},
\]
as $n\rightarrow\infty$.

\section{Probability of an essential return reaching a certain
depth} \label{sec:Depth Probability}

In the study of extremes, one is mostly interested in the
probability of occurrence of exceedances of the level $u_n$. As we
have already mentioned in Section~\ref{sec:strategy}, these events
are related with the occurrence of deep returns. Thus, in this
section we do some preparatory work by estimating the probability of
the returns hitting a given depth.

For each $x\in I$, let $v_n(x)$ denote the number of essential
return situations of $x$ between $1$ and $n$, $s_n(x)$ be the number
of those which are actual essential return times and
$\mathfrak{S}_n$ the number of the latter that correspond to
\emph{deep essential returns} of the orbit of $x$, i.e, with return
depths above a threshold $\Theta\geq \Delta$. Observe that
$v_n(x)-s_n(x)$ is the exact number of escaping situations of the
orbit of $x$, up to $n$.

Given the integers $0\leq s\leq \frac{2n}{\Theta}$, $s\leq v\leq n$
and the $s$ integers $\gamma_1,\ldots,\gamma_{s}$, each greater than
or equal to $\Theta$, we define the event:

\[
A_{\gamma_1,\ldots,\gamma_s}^{v,s}(n)=\left\{x\in I
\;:\;v_n(x)=v,\,\mathfrak{S}_n(x)=s\, \mbox{\begin{tabular}[t]{c}
and the depth of the i-th deep essen- \\ tial return is $\gamma_i$
$\forall i\in \{1,\ldots,s\}$\end{tabular}} \right\}.
\]
\begin{remark}
\label{rem:upper-bound-sn} Observe that the upper bound
$\frac{2n}{\Theta}$ for the number of deep essential returns up to
time $n$ derives from the fact that each deep essential return
originates a bound period of length at least $\tfrac{1}{2}\Theta$
(see Section \ref{subsec:bound-period-estimates}). Since during the
bound periods there cannot be any essential return, the number of
deep essential returns occurring in a period of length $n$ is at
most $\frac{n}{\frac{1}{2}\Theta}$.
\end{remark}

\begin{proposition}
\label{prop:depth-probability-1}Given the integers $0\leq s\leq
\frac{2n}{\Theta}$ and $s\leq v\leq n$, consider $s$ integers
$\gamma_1,\ldots,\gamma_{s}$, each greater than or equal to
$\Theta$. If $\Theta$ is large enough, then
\[
\lambda\left(A_{\gamma_1,\ldots,\gamma_s}^{v,s}(n)\right)\leq
\binom{v}{s}
\mbox{Exp}\left\{-(1-6\beta)\sum_{i=1}^s\gamma_i\right\}.
\]
\end{proposition}

\begin{proof}
Fix $n\in\N$ and take $\omega_0\in \mathcal{P}_0$. Note that the
functions $v_n$, $s_n$ and $\mathfrak{S}_n$ are constant on each
$\omega\in \mathcal{P}_n$. Let $\omega\in \omega_0\cap
\mathcal{P}_n$ be such that $v_n(\omega)=v$. Then, there is a
sequence $1\leq t_1\leq\ldots\leq t_v\leq n$ of essential return
situations. Let $\omega_i$ denote the element of the partition
$\mathcal{P}_{t_i}$ that contains $\omega$. We have
$\omega_0\supset\omega_1\supset\ldots\supset\omega_v=\omega$. For
each $j\in\{0,\ldots,v\}$ we define the set:
\[
Q_j=\bigcup_{\omega\in\mathcal P_n\cap\omega_0}\omega_j,
\]
and its partition
\[
\mathcal{Q}_j=\{\omega_j:\,\omega\in\mathcal P_n\cap\omega_0\}.
\]
Let $\omega\in\mathcal P_n$ be such that $\mathfrak{S}_n(\omega)=s$.
Then, we may consider $1\leq r_1\leq\ldots\leq r_s\leq v$ with $r_i$
indicating that the $i$-th deep essential return occurs in the
$r_i$-th essential return situation. Now, set $V(0)=Q_0=\omega_0$.
Fix $s$ integers $1\leq r_1\leq\ldots\leq r_s\leq v$. Next, for each
$j\leq v$ we define recursively the sets $V(j)$. Although the set
$V(v)$ will depend on the fixed integers $1\leq r_1\leq\ldots\leq
r_s\leq v$, we do not indicate this so that the notation is not
overloaded.
 Suppose that
$V(j-1)$ is already defined and $r_{i-1}<j<r_i$. Then, we set
\[
V(j)=\bigcup_ {\omega\in\mathcal Q_j}\omega\cap
f_a^{-t_j}(I-U_\Theta)\cap V(j-1).
\]
If $j=r_i$ then we define
\[
V(j)=\bigcup_ {\omega\in\mathcal Q_{j}}\omega\cap
f_a^{-t_{j}}(I_{\gamma_i}\cup I_{-\gamma_i})\cap V(j-1).
\]
 Observe that for every
 $j\in\{1,\ldots,v\}$ we have
$\frac{|V(j)|}{|V(j-1)|}\leq 1.$ Therefore, we concentrate in
finding a better estimate for $\frac{|V(r_i)|}{|V(r_i-1)|}$.
Consider that $\omega_{r_i}\in\mathcal Q_{r_i}\cap V(r_i)$ and let
$\omega_{r_i-1}\in\mathcal Q_{r_i-1}\cap V(r_i-1)$ contain
 $\omega_{r_i}$.
We have to consider two situations depending on whether $t_{r_i-1}$
is an escaping situation or an essential return.

Let us suppose first that $t_{r_i-1}$ was an essential return with
return depth $\eta$. Then,
\begin{align*}
\frac{|\omega_{r_i}|}{|\omega_{r_i-1}|}&\leq
\frac{|\omega_{r_i}|}{|\widehat{\omega}_{r_i-1}|},\;\text{ where
$\widehat{\omega}_{r_i-1}=
\omega_{r_i-1}\cap f_a^{-t_{r_i}}(U_1)$} \\
&\leq C\frac{\left|f_a^{t_{r_i}}(\omega_{r_i})\right|}
{\left|f_a^{t_{r_i}}(\widehat{\omega}_{r_i-1})\right|},\; \text{by
\eqref{eq:bounded-distortion}}
\\%lema referente ao bounded distortion
&\leq C\frac{2\e^{-\gamma_i}}{\e^{-5\beta\eta}},\;\text{by
\eqref{eq:return-after-essential}}.\\
 %lema 5.2-a) (ii) da tese do FJ
\end{align*}
Note that when $r_{i-1}=r_i-1$ then $\eta=\gamma_{i-1}$. If, on the
other hand, $r_{i-1}<r_i-1$ then $t_{r_i-1}$ is an essential return
with depth $\eta<\Theta\leq \gamma_{i-1}$. Thus, in both situations,
 we have
\[
\frac{|\omega_{r_i}|}{|\omega_{r_i-1}|}\leq
2C\frac{\e^{-\gamma_i}}{\e^{-5\beta\gamma_{i-1}}}.
\]
When $t_{r_i-1}$ is an escape situation, instead of using
\eqref{eq:return-after-essential}, we can use
\eqref{eq:return-after-escape} and obtain
\[
\frac{|\omega_{r_i}|}{|\omega_{r_i-1}|}\leq
2C\frac{\e^{-\gamma_i}}{\e^{-\beta\Delta}}\leq
2C\frac{\e^{-\gamma_i}}{\e^{-5\beta\gamma_{i-1}}}.
\]
Observe also that if $\widehat{\omega}_{r_i-1} \neq\omega_{r_i-1}$
then, because we are assuming that $\omega_{r_i}\neq\emptyset$, we
have $\left|f_a^{t_{r_i}} (\widehat{\omega}_{r_i-1})\right|\geq
\e^{-1}-\e^{-\Theta}\geq \e^{-5\beta\gamma_{i-1}}$, for large
$\Theta$.

At this point we may write
\begin{align*}
|V(r_i)|&= \sum_{\omega_{r_i}\in \mathcal Q_{r_i}\cap V(r_i)}
\frac{|\omega_{r_i}|}{|\omega_{r_i-1}|}
|\omega_{r_i-1}| \\
                       &\leq 2C\e^{-\gamma_i}\e^{5\beta\gamma_{i-1}}
\sum_{\omega_{r_i}\in
\mathcal Q_{r_i}\cap V(r_i)}|\omega_{r_i-1}|\\
    &\leq
    2C\e^{-\gamma_i}\e^{5\beta\gamma_{i-1}}
    |V(r_i-1)|.
\end{align*}
This yields
\[
|V(v)|\leq
(2C)^s\mbox{Exp}\left\{-(1-5\beta)\sum_{i=1}^s\gamma_i\right\}
\e^{5\beta\gamma_0} |V(0)|,
\]
where $\gamma_0$ is given by the interval $\omega_0\in
\mathcal{P}_0$. If $\omega_0=I_{(\eta_0,k_0)}$ with $|\eta_0|\geq
\Delta$ and $1\leq k_0\leq \eta_0^2$, then $\gamma_0=|\eta_0|$. If
$\omega_0=(\delta,1]$ or $\omega_0=[-1,-\delta)$, then we can take
$\gamma_0=0$.

Now, we have to take into account the number of possibilities of
having the occurrence of the event $V(v)$ implying the occurrence of
the event $A_{\gamma_1,\ldots,\gamma_s}^{v,s}(n)$. The number of
possible configurations related with the different values that the
integers $r_1,\ldots r_s$ can take is $\binom{v}{s}$. Hence, it
follows that

\begin{align*}
\lambda\left(A_{\gamma_1,\ldots,\gamma_s}^{v,s}(n)\right)&\leq
(2C)^s\binom{v}{s}\mbox{Exp}\left\{-(1-5\beta)
            \sum_{i=1}^s\gamma_i\right\}
\sum_{\omega_o\in\mathcal{P}_0}\e^{5\beta|\gamma_0|}|\omega_0|\\
&\leq (2C)^s\binom{v}{s}\mbox{Exp}\left\{-(1-5\beta)
            \sum_{i=1}^s\gamma_i\right\}\left(2(1-\delta)+
            \sum_{|\eta_0|\geq\Delta}\e^{5\beta\eta_0}
            \e^{-|\eta_0|}\right)\\
&\leq 3(2C)^s \binom{v}{s}\mbox{Exp}\left\{-(1-5\beta)
            \sum_{i=1}^s\gamma_i\right\},\quad\text{for
$\Delta$
large enough}\\
&\leq \binom{v}{s}\mbox{Exp}\left\{-(1-6\beta)
            \sum_{i=1}^s\gamma_i\right\}.
\end{align*}
The last inequality results from the fact that $s\Theta\leq
\sum_{i=1}^s\gamma_i$ and the freedom to choose a sufficiently large
$\Theta$.
\end{proof}

Given the integers $0\leq s\leq \frac{2n}{\Theta}$, $s\leq v\leq n$
and the integers $\gamma_0,\gamma_{1}$, both greater than or equal
to $\Theta$, we consider the event:
\[
B_{\gamma_0,\gamma_{1}}^{v,s}(n)=\left\{x\in I_{\gamma_0}
\;:\;v_n(x)=v,\,\mathfrak{S}_n(x)=s\, \mbox{\begin{tabular}[t]{c}
and $n$ is a free
deep return \\
 with depth $\gamma_{1}$ \end{tabular}} \right\}.
\]

\begin{corollary}
\label{cor:prob-deep-returns} Consider the integers $0\leq s\leq
\frac{2n}{\Theta}$, $s\leq v\leq n$ and $\gamma_0,\gamma_{1}\geq
\Theta$. If $\Theta$ is large enough, then
\[
\lambda\left(B_{\gamma_0,\gamma_1}^{v,s}(n)\right)\leq \binom{v}{s}
\mbox{Exp}\left\{-(1-6\beta)(\gamma_0+\gamma_1)\right\}.
\]
\end{corollary}

The proof of this statement follows easily from
Proposition~\ref{prop:depth-probability-1}, by observing that,
although $n$ may be an inessential deep return time (instead of an
essential deep return), the estimates still prevail and for
$\Theta>\Delta$ large enough we have
$\sum_{\gamma\geq\Theta}\e^{-(1-6\beta)\gamma}\leq 1$.

\section{The condition $D'(u_n)$}
\label{sec:D'un}

Assume that $X_0,X_1,\ldots$ is the stationary stochastic process
defined in \eqref{eq:def-process} with common d.f. $G_a$. For
$z\in[-1,1]$ the event $\{X_j(z)>u_n\}$ corresponds to the set
$f_a^{-j}\left((u_n,1]\right)$. If $n$ is sufficiently large, then
$f_a^{-1}(u_n,1]=(-\sqrt{(1-u_n)/a}, \sqrt{(1-u_n)/a}) \subset
U_{\Delta}$. We may define
\begin{equation}
\label{eq:def-Theta-1}
\Theta=\Theta(n)=\left[-\frac12\log\frac{1-u_n}{a}\right].
\end{equation}
This way, if an exceedance occurs at time $j$ then a deep return
with depth over the threshold $\Theta$ must have happened at time
$j-1$, i.e., if $X_j(z)>u_n$ then $X_{j-1}(z)=f_a^{j-1}(z)\in
U_{\Theta}$.

 Remember that the
sequence $u_n$ is such that $n(1-G_a(u_n))\rightarrow \tau$, as
$n\rightarrow \infty$, which we rewrite as $1-G_a(u_n)\asymp1/n$.
Then, by \eqref{eq:tail-g1}, we get $u_n\asymp1-1/n^2$, which
according to \eqref{eq:def-Theta-1} leads to
\begin{equation}\label{eq:estimate-Theta}
\Theta\asymp\log n,
\end{equation} meaning that $\tfrac{\Theta}{\log n}\rightarrow c$,
for some
$c>0$, as $n\rightarrow\infty$.

Observe that we are dealing with very small perturbations of $f_2$
for which $f_2^j(0)=-1$ for every $j\geq 2$. Thus, one expects that
after a deep return to the critical region (a tight vicinity of $0$)
it should take a considerable amount of time before another deep
return should occur. Since exceedances are related with the
occurrence of deep returns then one may have a fair amount of belief
that condition \eqref{eq:D'un} holds for the sequence
$X_0,X_1,\ldots$

\begin{remark}
\label{rem:independent-case} If the sequence $X_0,X_1,\ldots$ was
independent, then \eqref{eq:D'un} would follow easily since
\begin{align*}
n\sum_{j=1}^{[n/k]}P\{X_0>u_n\text{ and
}X_j>u_n\}&=n\sum_{j=1}^{[n/k]}P\{X_0>u_n\}P\{X_j>u_n\}=
n\sum_{j=1}^{[n/k]} (1-G_a(u_n))^2\\
&\leq \frac{n^2}{k}
(1-G_a(u_n))^2\xrightarrow[n\rightarrow\infty]{}\frac{\tau^2}{k}
\xrightarrow[k\rightarrow\infty]{}0.
\end{align*}

\end{remark}

Let us give some insight into the argument we use to prove that
\eqref{eq:D'un} holds for $X_0,X_1,\ldots$
\begin{enumerate}

\item We use the exponential decay of correlations (see
\eqref{eq:decay-correlations}) to compute a turning instant $T=T(n)$
such that the dependence between $X_0$ and $X_j$ with $j>T$ is
negligible. This suggests the splitting of the time interval
$\{1,\ldots,[n/k]\}$ into $\{1,\ldots,T\}$ and
$\{T+1,\ldots,[n/k]\}$, when $n$ is sufficiently large.

\item During the time interval $\{T+1,\ldots,[n/k]\}$ we use the fact
that for $j> T$ the r.v. $X_j$ is almost independent of $X_0$ and
argue like in Remark~\ref{rem:independent-case}.

\item For $j\in\{1,\ldots,T\}$ we appeal to
Corollary~\ref{cor:prob-deep-returns} to bound $P\{X_0>u_n\text{ and
}X_j>u_n\}$ and afterwards we use the fact that, for $n$ large,
$T\ll[n/k]$ to finish the proof.

\end{enumerate}

\vspace{0.3cm} \noindent\textbf{Step (1)} \vspace{0.2cm}

 Taking
$\varphi=\psi=\I_{(u_n,1]}$ in \eqref{eq:decay-correlations}, we get
\begin{align*}
\left|^{} \mu_a\{X_0>u_n\mbox{ and
}\right.&\left.X_j>u_n\}-\left[\mu_a\{X_0>u_n\}\right]^2\right|=\\&=
\left|\int \I_{(u_n,1]}\cdot\I_{(u_n,1]}\circ f_a^jd\mu_a-\left(\int
\I_{(u_n,1]}d\mu_a\right)^2\right|\\&\leq C\varsigma^j,
\end{align*}
where we may assume that $C$ is the same for all $n\in\N$, because
$||\I_{(u_n,1]}||_\infty=1$ and the total variation of
$\I_{(u_n,1]}$ is equal to $1$, for every $n\in\N$.

We compute $T=T(n)$ such that for every $j\geq T$ we have
\[
C\varsigma^j<\frac{1}{n^3}.
\]
Since $C\varsigma^j<\frac{1}{n^3}\Leftrightarrow
j>\frac{1}{\log\varsigma^{-1}}(3\log n+\log C)$, we simply take, for
$n$ sufficiently large,
\begin{equation}
\label{eq:choice-T} T=\frac{4}{\log\varsigma^{-1}}\log n.
\end{equation}
For fixed $k$ and $n$ sufficiently large, we have that $T<[n/k]$.
Hence, we may write
\begin{multline*}
n\sum_{j=1}^{[n/k]}P\{X_0>u_n\mbox{ and }X_j>u_n\}=\\
=n\sum_{j=1}^{T}P\{X_0>u_n\mbox{ and
}X_j>u_n\}+n\sum_{j=T+1}^{[n/k]}P\{X_0>u_n\mbox{ and }X_j>u_n\}.
\end{multline*}
In step (2) below we deal with the second term in the sum, leaving
the first term for step (3).

\vspace{0.3cm} \noindent\textbf{Step (2)} \vspace{0.2cm}

Let us show that
$\limsup_{n\rightarrow\infty}n\sum_{j=T+1}^{[n/k]}P\{X_0>u_n\mbox{
and }X_j>u_n\}\rightarrow 0$, as $k\rightarrow \infty$. By choice of
$T$ we have
\begin{align*}
n\sum_{j=T+1}^{[n/k]}P\{X_0>u_n\mbox{ and }X_j>u_n\}&\leq n\left
((1-G_a(u_n))^2+\frac{1}{n^3}\right)[n/k]\\
&\leq \frac{n^2}k(1-G_a(u_n))^2+\frac{n^2}{kn^3}.
\end{align*}
Now,
$\frac{n^2}k(1-G_a(u_n))^2+\frac{n^2}{kn^3}\xrightarrow[n\rightarrow\infty]{}
\frac{\tau^2}k\xrightarrow[k\rightarrow\infty]{}0$ and the result
follows.

\vspace{0.3cm} \noindent\textbf{Step (3)} \vspace{0.2cm}

We are left with the burden of controlling the term
$n\sum_{j=1}^{T}P\{X_0>u_n\mbox{ and }X_j>u_n\}$. We begin with the
following lemma that will enable us to bound the number of
exceedances occurring during the time period $\{1,\ldots,T\}$. In
what follows, we are always assuming that $n$ is large enough so
that $\Theta>\Delta$.
\begin{lemma}
\label{lem:time-between-exceedances} If a deep return occurs at time
$t$ (with depth over the threshold $\Theta$), then the next deep
return can only occur after $t+\Theta/2$.
\end{lemma}
\begin{proof}
For every $x\in U_\Theta$, the bound period associated to $x$ is
such that $p(x)\geq \Theta/2$, by
Section~\ref{subsec:bound-period-estimates}~
\eqref{item:bound-period-bounds-1d}. For all $j\leq [\Theta/2]$ we
have
\begin{align*}
\left|f_a^{j}(x)\right|&\geq \left|f_a^{j}(0)\right|-\e^{-\beta
j}\overset{\eqref{eq:BA}}{\geq}\e^{-\alpha \sqrt j}-\e^{-\beta
j}\geq \e^{-\alpha j}\left(1-\e^{(\alpha-\beta)j}\right)\\
&\geq \e^{-\alpha j}\left(1-\e^{(\alpha-\beta)}\right),\,
\text{since $\alpha-\beta<0$}\\
&\geq \e^{-\alpha \Theta/2}\left(1-\e^{(\alpha-\beta)}\right),\,
\text{since $j\leq\Theta/2$}\\
&\geq \e^{-\alpha\Theta},\, \text{if
$n$ is large enough so that $1-\e^{\alpha-\beta}
\geq \e^{-\alpha\Theta/2}$}\\
&\geq \e^{-\Theta},\, \text{since $\alpha<1$.}
\end{align*}\end{proof}
As a consequence of Lemma~\ref{lem:time-between-exceedances} we have
that the maximum number of exceedances up to time $T$ is at most
$2T/\Theta$.

Next lemma shows that in the bound period following a deep return
over the level $\Theta$, there cannot occur deep bound returns
before time $T$.
\begin{lemma}
\label{lem:no-bound-retruns} For every $a\in\B$, if $n$ is
sufficiently large then for every $x\in U_\Theta$ we have that
$f_a^j(x)\notin U_\Theta$ for all $j\leq \min\{p(x),T\}$.
\end{lemma}
\begin{proof}
Consider the map $h:[1,\infty)\to\R$ given by $h(y)=\e^{-\alpha\sqrt
y}-\e^{-\beta y}$. There is $J=J(\alpha,\beta)\in\N$ such that
$h'(y)<0$ for every $y\geq J$. Let $D=\e^{-\alpha\sqrt J}$.

Having in mind that by \eqref{eq:estimate-Theta} and
\eqref{eq:choice-T} we have $\Theta=\Theta(n)\asymp\log n$ and
$T\asymp\log n$, respectively, assume that $n$ is large enough so
that:

\begin{enumerate}
  \item \label{item:proof-cond-1}$\e^{-\Theta}<D/2$;

  \item \label{item:proof-cond-2}by continuity of $f_a$, for every $x\in U_\Theta$ and all
  $j\leq J$ we have $|f_a^j(x)-f_a^j(0)|<D/2$;

  \item \label{item:proof-cond-3}$\e^{-\alpha
\sqrt T}-\e^{-\beta T}>\e^{-\Theta}$.
\end{enumerate}

Let $x\in U_\Theta$ and consider $j\leq \min\{p(x),J\}$. By
assumption \eqref{item:proof-cond-2} we have
$|f_a^j(x)-f_a^j(0)|<D/2$. Using \eqref{eq:BA} and hypothesis
\eqref{item:proof-cond-1} we get
\[
|f_a^j(x)|\geq |f_a^j(0)|-|f_a^j(x)-f_a^j(0)|\geq D-D/2\geq
D/2>e^{-\Theta}.
\]

Now, let $x\in U_\Theta$ and consider $J<j\leq \min\{p(x),T\}$.
Using the definition of bound period, \eqref{eq:BA} and condition
\eqref{item:proof-cond-3} we have
\[
|f_a^j(x)|\geq |f_a^j(0)|-|f_a^j(x)-f_a^j(0)|\geq \e^{-\alpha\sqrt
j}-\e^{-\beta j}\geq \e^{-\alpha\sqrt T}-\e^{-\beta T}>
\e^{-\Theta},
\]
where the third inequality derives from the fact that on
$[J,\infty)$ the function $h$ is decreasing.
\end{proof}
As a consequence, if there is a deep return at time $t$, then we
cannot have bound returns during the time period $(t,t+T]$. This
way, we may use Corollary~\ref{cor:prob-deep-returns} to estimate
$|U_{\Theta}\cap f_a^{-j}U_\Theta|$, for $j\leq T$.

Observe that for $n$ large enough, we have
\begin{align*}
P\{X_0>u_n\text{ and }X_j>u_n\}&=\mu_a\left\{(u_n,1]\cap
f_a^{-j}(u_n,1]\right\}\\&=\mu_a \left\{f_a^{-1}(u_n,1]\cap
f_a^{-(j+1)}(u_n,1]\right\},
\end{align*}
and $f_a^{-1}(u_n,1]\cap f_a^{-j}\left(f_a^{-1}(u_n,1]\right)\subset
U_{\Theta}\cap f_a^{-j}U_{\Theta}.$  In what follows we will use
``const'' to denote several positive constants independent of $n$.

Note that
\[
U_{\Theta}\cap
f_a^{-j}U_\Theta\subset\bigcup_{s=0}^{2j/\Theta}\bigcup_{v=s}^j
\bigcup_{\gamma_0\geq \Theta} \bigcup_{\gamma_1\geq \Theta}
B_{\gamma_0,\gamma_1}^{v,s}(j).
\]
Hence, by Corollary~\ref{cor:prob-deep-returns}, we have
\begin{align*}
|U_{\Theta}\cap f_a^{-j}U_\Theta|&\leq
\sum_{s=0}^{2j/\Theta}\sum_{v=s}^j\sum_{\gamma_0\geq \Theta}
\sum_{\gamma_1\geq \Theta}\binom{v}{s}
\e^{-(1-6\beta)(\gamma_0+\gamma_1)}\leq
\mbox{const}\sum_{s=0}^{2T/\Theta}\sum_{v=s}^T\binom{v}{s}
\e^{-(1-6\beta)2\Theta}\\
&\leq \mbox{const}\sum_{s=0}^{2T/\Theta}\sum_{v=s}^T\binom{T}{s}
\e^{-(1-6\beta)2\Theta}\leq
\mbox{const}\sum_{s=0}^{2T/\Theta}T\binom{T}{s}
\e^{-(1-6\beta)2\Theta}.
\end{align*}
At this point, we estimate $2T/\Theta$. Recalling
\eqref{eq:estimate-Theta} and \eqref{eq:choice-T}, one easily gets
that there exists a constant $C_1>0$ such that $2T/\Theta\leq C_1$,
for $n$ sufficiently large. So, to proceed with the estimation
$|U_{\Theta}\cap f_a^{-j}U_\Theta|$, assume that $n$ is sufficiently
large so that $2T/\Theta\leq C_1$ and $C_1\ll T$. Then,
\begin{align*}
|U_{\Theta}\cap f_a^{-j}U_\Theta|&\leq
\mbox{const}\sum_{s=0}^{2T/\Theta}T\binom{T}{s}
\e^{-(1-6\beta)2\Theta}\leq
\mbox{const}\sum_{s=0}^{2T/\Theta}T\binom{T}{C_1}
\e^{-(1-6\beta)2\Theta}\\
&\leq \mbox{const}\frac{2T}{\Theta}T\binom{T}{C_1}
\e^{-(1-6\beta)2\Theta}\leq \mbox{const}T^{C_1+1}
\e^{-(1-6\beta)2\Theta} .
\end{align*}

So far we managed to obtain an estimate for $|U_{\Theta}\cap
f_a^{-j}U_\Theta|$. Let us see how too derive one for
$P\{X_0>u_n\text{ and }X_j>u_n\}\leq\mu_a\left\{U_{\Theta}\cap
f_a^{-j}U_\Theta\right\}$. By assertion~\eqref{item:acim-density-1}
in Section~\ref{subsection:existence-inv-measures} we have
\begin{align*}
\mu_a\left\{U_{\Theta}\cap
f_a^{-j}U_\Theta\right\}&=\int_{U_{\Theta}\cap
f_a^{-j}U_\Theta}\rho_1^a(x)dx+\int_{U_{\Theta}\cap
f_a^{-j}U_\Theta}\rho_2^a(x)dx\\
&\leq \mbox{const}.|U_{\Theta}\cap
f_a^{-j}U_\Theta|+\sum_{j=1}^{\Theta^2}\int_{U_{\Theta}\cap
f_a^{-j}U_\Theta}\tfrac{(1.9)^{-j}}{\sqrt{|x-f^j_a(0)|}}dx\\
&\quad+\sum_{j=\Theta^2}^{\infty} \int_{U_{\Theta}\cap
f_a^{-j}U_\Theta}\tfrac{(1.9)^{-j}}{\sqrt{|x-f^j_a(0)|}}dx.
\end{align*}
Using \eqref{eq:BA} we get $|f_a^j(0)|>\e^{-\alpha\Theta}$, for all
$j\leq\Theta^2$. Thus, for all $x\in U_{\Theta}$ and $\Theta$ large
enough, we have
$\sqrt{|x-f^j_a(0)|}>\sqrt{\e^{-\alpha\Theta}-\e^{-\Theta}}>
1/2.\e^{-\alpha\Theta/2}$, which implies that
$\tfrac1{\sqrt{|x-f^j_a(0)|}}\leq2\e^{\alpha\Theta/2}\leq2
\e^{\beta\Theta}$ (recall that $\beta>\alpha$). Consequently, for
sufficiently large $\Theta$,
\begin{align*}
\mu_a\left\{U_{\Theta}\cap f_a^{-j}U_\Theta\right\}&\leq
\mbox{const}.|U_{\Theta}\cap
f_a^{-j}U_\Theta|+\mbox{const}.|U_{\Theta}\cap
f_a^{-j}U_\Theta|.\e^{\beta\Theta}+\mbox{const}.\e^{\Theta^2\log1.9}\\
&\leq \mbox{const}.T^{C_1+1} \e^{-(1-7\beta)2\Theta}.
\end{align*}

Now, since $\e^{-\Theta}\leq \mbox{const}(1-G_a(u_n))$, we finally
conclude:
\begin{align*}
n\sum_{j=1}^TP\{X_0>u_n\text{ and }X_j>u_n\}&\leq \mbox{const}\cdot
n\sum_{j=1}^TT^{C_1+1}(1-G_a(u_n))^{2(1-7\beta)}\\
&\leq \mbox{const}\cdot n T^{C_1+2}(1-G_a(u_n))^{2(1-7\beta)}\\
&\leq \mbox{const}\cdot n (\log
n)^{C_1+2}(1-G_a(u_n))^{2-14\beta}\\
&\leq \mbox{const}\cdot [n(1-G_a(u_n))]^{3/2}
(1-G_a(u_n))^{1/2-14\beta},
\end{align*}
for sufficiently large $n$. The result follows because
\[
\lim_{n\rightarrow\infty} [n(1-G_a(u_n))]^{3/2}
(1-G_a(u_n))^{1/2-14\beta}=\tau^{3/2}\cdot0=0.
\]

\section{Simulation Study}
\label{sec:simulation}

In this section we present a small simulation study illustrating the
finite sample behavior of the normalized $M_n$, defined in
\eqref{eq:def-maximum}, for the Benedicks-Carleson quadratic map
$f_2(x)=1-2x^2$. We have that $P\{X_0\leq
x\}=G_2(x)=1/2+\arcsin(x)/\pi$. According to
Theorem~\ref{th:A-tipo-1}, the normalizing sequence is
$a_n=(1-\cos(\pi/n))^{-1},$ for each $n\in\N,$ and the theoretical
limiting distribution for $a_n(M_n-1)$ is
$$H(x)=\begin{cases}
\e^{-(-x)^{1/2}}&,x\leq0\\
1&,x>0
\end{cases}.$$

We performed the following experiment. We picked at random
(according to the d.f. $G_2$) a point $z$ in the interval $[-1,1]$,
computed its orbit up to time $n$ and calculated $M_n(z)=\max\{z,
f_2(z), \ldots, f_2^{n-1}(z)\}$. We repeated the process $m$ times
to obtain a sample $\{M_n(z_1), \ldots,M_n(z_m)\}$ and approximated,
for certain values of $x$, $P\{a_n(M_n-b_n)\leq x\}$ by
\begin{equation}\label{eq:simul-approx}
\frac1m\sum_{i=1}^m\I_{\{a_n(M_n(z_i)-b_n)\leq x\}},
\end{equation}
where $\I_{\{a_n(M_n(z_i)-b_n)\leq x\}}=\begin{cases}
1&\mbox{if }a_n(M_n(z_i)-b_n)\leq x\\
0&\mbox{if }a_n(M_n(z_i)-b_n)> x
\end{cases}$, for each $1\leq i\leq m$.

In Table~\ref{tab:simulation} we present the results obtained by
realizing the above experiment, considering different values of $x$,
$n$ and $m$ and compare them with the theoretical limiting ones
given by $H(x)$.

\begin{table}[htbp]
  \centering
\begin{tabular}{|c|c|c|c|c|c|c|}\hline
 & &\multicolumn{2}{|c|}{$n=1000$}&\multicolumn{2}{|c|}{$n=10000$}&
$n=20000$\\ \cline{3-7}
$x$&$H(x)$&$m=1000$&$m=10000$&$m=10000$&$m=20000$&$m=20000$\\\hline
-0.001 & 0.9689 & 0.976 & 0.9671 & 0.9719 & 0.9677 & 0.9708\\
\hline

-0.01 & 0.9048 & 0.894 & 0.9079 & 0.9056 & 0.9076 &0.9052\\ \hline

-0.1 & 0.7289 & 0.724 & 0.7263 & 0.7323& 0.7303 & 0.7335\\ \hline

-0.3 & 0.5783 & 0.569 & 0.5773 & 0.5823 & 0.5840 & 0.5813\\
\hline

-0.5 & 0.4931 & 0.491 & 0.4906 & 0.5012 & 0.4984 & 0.4941\\ \hline

-0.7 & 0.4332 & 0.407 & 0.4272 & 0.4403 & 0.4374 & 0.4338\\ \hline

-1 & 0.3679 & 0.388 & 0.3631 & 0.3663 & 0.3729 & 0.3678\\ \hline

-3 & 0.1769 & 0.164 & 0.1731 & 0.1748 & 0.1833 & 0.1729\\ \hline

-5 & 0.1069 & 0.124 & 0.1024 & 0.1092 & 0.1108 & 0.1056\\ \hline

-8 & 0.0591 & 0.059 & 0.0510 & 0.0557 & 0.0617 & 0.0580\\ \hline

-10 & 0.0423 & 0.049 & 0.0350 & 0.0435 & 0.0438 & 0.0414\\ \hline

-30 & 0.0042 & 0.002 & 0.0031 & 0.0033 & 0.0048 & 0.0041\\
\hline

-50 & 0.00085 & 0.001 & 0.0007 & 0.0009 & 0.0007 & 0.0009\\ \hline

\end{tabular}
 \caption{Simulation results}
  \label{tab:simulation}
\end{table}

As one may verify the results of the experiment are quite close to
the asymptotic theoretical ones and there is a general tendency of
getting better as $n$ increases which is precisely the behavior we
were expecting. It is also noticeable that there is an improvement
when $m$ increases, which is also predictable since our
approximation \eqref{eq:simul-approx} gets to be more accurate.

\subsection*{Acknowledgments}
We wish to thank M. P. Carvalho for valuable observations and
relevant suggestions. We are grateful to J. F. Alves, M. A. Brito
and R. Leadbetter for fruitful conversations and comments. We are
also grateful to an anonymous referee for his comments and important
suggestions that helped to improve the paper.

\end{document}